\documentclass[11pt]{amsart}
\usepackage{amssymb, latexsym, amsmath, amsfonts, hyperref, enumitem, fancyhdr}
\usepackage{amsmath,amscd}
\usepackage{cleveref}
\usepackage{mathrsfs}
\newtheorem{thm}{Theorem}[section]
\newtheorem{cor}[thm]{Corollary}
\newtheorem{lem}[thm]{Lemma}

\theoremstyle{definition}
\newtheorem{defn}[thm]{Definition}
\theoremstyle{remark}
\newtheorem{rem}[thm]{Remark}
\numberwithin{equation}{section}

\hypersetup{
  colorlinks,
  citecolor=black,
  linkcolor=black,
  urlcolor=blue}
\theoremstyle{plain}

  {\end{enumerate}}

\hypersetup{
  colorlinks,
  citecolor=blue,
  linkcolor=blue,
  urlcolor=blue}

\begin{document}
\title{On Multivariate Matsaev's Conjecture}

\thanks{The named author acknowledges Council for Scientific and Industrial Research, MHRD, Government of India for financial support during this work.}
%
\author{Samya Kumar Ray}


\address{Samya Kumar Ray: School of Mathematics and Statistics,Wuhan University, Wuhan-430072, China}
\email{samyaray7777@gmail.com}

%

\pagestyle{headings}

\begin{abstract} In this article, we study multivariate generalizations of Matsaev's conjecture in commutative and non-commutative $L^p$-spaces. We prove that the multivariate analogue of Matsaev's conjecture is eventually false for all $1<p<\infty.$ We exhibit various joint dilation results on non-commutative $L^p$-spaces.
\end{abstract}

\maketitle

\section{Introduction and Main Results}
In 1950, von Neumann \cite{vN50} proved that if $T$ is a contraction on a Hilbert space and $P$ is any complex polynomial in a single variable, then
\begin{equation*}\label{oo}
\|P(T)\|\leq\|P\|_{\infty,\mathbb{D}},
\end{equation*}where for any complex polynomial $P$ in $n$-variables, and any $K\subseteq \mathbb{C}^n,$ we define  \[\|P\|_{\infty,K}\colon=\sup\{|P(z_1,\dots,z_n)|\colon(z_1,\dots,z_n)\in K\}.\]
The von Neumann inequality has turned out to be one of the most important operator theoretic tools (see \cite{SzF70}). The popular way of proving the von Neumann inequality is through the dilation theory, which ensures that any contraction on a Hilbert space always admits a unitary dilation (see \cite{PA02} Chapter 3 and \cite{SzF70} Chapter 1). Ando's dilation theorem \cite{AN63} readily implies that the von Neumann inequality holds true for two commuting contractions. Unfortunately, the von Neumann inequality fails to be true in general. In 1974, Varopoulos \cite{VA74} used sophisticated probabilistic tools combined with techniques from the metric theory of tensor product to show that the von Neumann inequality does not generalize to arbitrary number of commuting contractions. In the end of the same paper, together with Kaijser (\cite{VA74}), he constructed an explicit example of three commuting contractions $T_1$, $T_2$, $T_3$ on a five dimensional complex Hilbert space and a polynomial $P$ of degree two such that
\begin{equation*}
\|P(T_1,T_2,T_3)\|>\|P\|_{\infty,\mathbb{D}^3}.
\end{equation*}
For more on the von Neumann inequality and counterexamples, we recommend the readers \cite{BAB13}, \cite{SzF70}, \cite{GUR18}, \cite{CRD75}, \cite{KN16},  \cite{DR83}), \cite{TO78} and \cite{HO01}.

In 1966, V. I. Matsaev (see \cite{NI74}) proposed following possible generalization of the von Neumann inequality on $L^p$-spaces, for $1\leq p\leq\infty,$ which is, given any contraction $T$ on a $\sigma$-finite $L^p$-space, and any polynomial $P$ in a single variable
\begin{equation*}
\|P(T)\|\leq\|P(R)\|_{\ell^p({\mathbb{N}})\to\ell^p({\mathbb{N}})},
\end{equation*}
where $R$ is the right shift operator on $\ell^p({\mathbb{N}})$, defined as $R(x_1,x_2,x_3,\dots)\colon=(0,x_1,x_2,x_3,\dots).$

The above mentioned proposition is renowned to be Matsaev's conjecture (see also \cite{PI01}, Page no. 30). Due to existence of isometric dilation, the conjecture remains true for contractions which admit a contractive positive majorant. However, Drury \cite{DR11} exhibited an explicit counterexample for $p=4$. We refer the reader \cite{AKS77}, \cite{COW76}, \cite{CORW77}, \cite{PE76}--\cite{PE85} for more information in this direction. Also, we recommend \cite{AR13} for generalizations in the setting of non-commutative $L^p$-spaces and \cite{FE97} for semigroups. We also recommend \cite{LE99}, \cite{FA15} for another formulation of Matsaev's conjecture in connection with analytic semigroup and functional calculus and \cite{ARL14}, \cite{AR16}, \cite{ARK17} \cite{ARFL17}, \cite{AR19}, \cite{AR191} for many results associated to Matsaev's conjecture and dilations on commutative and non-commutative $L^p$-spaces.

In this paper, we consider Matsaev's conjecture in various multivariate settings. We need the following notations.

Suppose $w=(w_{i_1,\ldots,i_n})_{i_j\in\mathbb{N}}$ is a sequence of positive real numbers and $X$ is a Banach space. Let us define
\[\ell^p(w,X)\colon=\big\{\big(x_{i_1,\ldots,i_n}\big)_{i_j\in\mathbb{N}}\colon x_{i_1,\ldots,i_n}\in X,\sum_{i_j\in\mathbb{N}}\big\|x_{i_1,\ldots, i_n}\big\|_{X}^pw_{i_1,\ldots,i_n}<\infty\big\},\]
with norm defined as $\big\|\big(x_{i_1,\ldots,i_n}\big)\big\|_{\ell^p(\omega,X)}\colon=\big(\sum_{i_j\in\mathbb{N}}\big\|x_{i_1,\ldots,i_n}\big\|_{X}^pw_{i_1,\ldots,i_n}\big)^{\frac{1}{p}}.$

If $w_{i_1,\ldots,i_n}=1$ for all ${i_j\in\mathbb{N},1\leq j\leq n}$, we denote the Banach space by $\ell^{p}({\mathbb{N}^n},X)$. We denote the Banach spaces $\ell^p(\omega,\mathbb{C})$ and $\ell^p({\mathbb{N}^n},\mathbb{C})$ by $\ell^p(\omega)$ and  $\ell^p(\mathbb{N}^n)$ respectively. 

\begin{defn}[Right shift operators on $\ell^p(w,X)$]\label{SH}
For $1\leq j\leq n$, the $j$-th right shift operator on $\ell^{p}(w,X)$ is denoted as $R_j$ and  is defined by

             \[
              (R_j(x))_{i_1,\dots,i_j,\dots,i_n}\colon=
              \begin{cases}
              0,& i_j=1\\
               x_{i_1,\dots,i_j-1,\dots,i_n},& i_j>1,
              \end{cases}
              \]
where $x=\big(x_{i_1,\ldots,i_n}\big)_{i_j\in\mathbb{N}}\in\ell^p(w,X).$
\end{defn}
\begin{defn}[Left shift operators on $\ell^p(w,X)$]\label{RH}
 For $1\leq j\leq n$, the $j$-th left shift operator on $\ell^{p}(w,X)$ is denoted as $S_j$ and  is defined by \[(S_j(x))_{i_1,\dots,i_j,\dots,i_n}\colon=x_{i_1,\dots,i_j+1,\dots,i_n}\] where $x=\big(x_{i_1,\ldots,i_n}\big)_{i_j\in\mathbb{N}}\in\ell^p(w,X).$
\end{defn}
For any commuting $n$-tuple of operators $\mathbf{T}=(T_1,\dots,T_n)$ on a Banach space $X$, let us denote $P(\mathbf{T})\colon=P(T_1,\dots,T_n),$ where $P$ is a polynomial in $n$-variables. Unless specified, we shall always work with $\sigma$-finite measure spaces and consider $1<p<\infty.$
\begin{defn}[$n$-Matsaev property]\label{MM}
	Let $\mathbf{T}=(T_1,\dots,T_n)$ be a commuting tuple of contractions on $L^p(\Omega,\mathbb{F},\mu)$. We say that the tuple $\mathbf{T}$ has $n$-Matsaev property if
	for all complex polynomials $P$ in $n$-variables,
	$$\|P(\mathbf{T})\|_{L^p(\Omega,\mathbb{F},\mu)\to L^p(\Omega,\mathbb{F},\mu)}\leq \|P(\mathcal{R})\|_{\ell^{p}(\mathbb{N}^n)\to \ell^{p}(\mathbb{N}^n)},$$
	
	where $\mathcal{R}=(R_1,\dots,R_n)$ is the $n$-tuple of right shift operators on $\ell^{p}(\mathbb{N}^n)$ as in \Cref{SH}.
\end{defn}
One can easily verify that in \Cref{MM}, $\ell^p(\mathbb N^n)$ and right shift operators on $\ell^p(\mathbb N^n)$ can be replaced by $\ell^p(\mathbb Z^n)$ and right shift operators on $\ell^p(\mathbb Z^n)$ respectively.

From well-known transference principle due to Coifman-Weiss \cite{COW76}, one can easily show that any commuting $n$-tuple of onto isometries has $n$-Matsaev property. In Section \eqref{COM}, we show that any $n$-tuple of commuting isometries (not necessarily onto) acting on $L^p$-space, has $n$-Matsaev property. However, adapting Peller's \cite{PE78} proof, we have following general theorem. 
\begin{thm}\label{T}
	Let $\mathbf{U}=(U_1,\dots,U_n)$ be a commuting tuple of isometries (not necessarily onto) on $L^p(\Omega,\mathbb{F},\mu)$ . Suppose $w=(w_{i_1\dots,i_n})$ is a sequence of positive real numbers indexed by $\mathbb{N}^n$ with the following properties:
	\begin{enumerate}\label{prope}
		\item[(1)]$w_{i_1,\dots,i_n}\geq1$ for all $i_j\in\mathbb{N},1\leq j\leq n$.
		
		\item[(2)]For each $1\leq j\leq n$, $\sup\limits_{i_j,1\leq j\leq n}\frac {w_{i_1,\dots,i_j+1,\dots,i_n}}{w_{i_1,\dots,i_j,\dots,i_n}}\colon=M_j$ is finite.\label{po}
	\end{enumerate}
	Then for any polynomial $P$ in $n$-variables,
	$$\|P(\mathbf{U})\|_{L^p(\Omega,\mathbb{F},\mu)\to L^p(\Omega,\mathbb{F},\mu)}\leq \|P(\mathcal{R})\|_{\ell^{p}(w,\mathbb{N}^n)\to \ell^{p}(w,\mathbb{N}^n)},$$
	
	where $\mathcal{R}=(R_1,\dots,R_n)$ is the $n$-tuple of right shift operators on $\ell^{p}(w,\mathbb{N}^n),$ defined as in \eqref{SH}.
\end{thm}

In this section, we also prove that the multivariate analogue of Matsaev's conjecture fails. More precisely, we establish the following. For a fixed $1<p<2,$ let $\mathcal{C}_p(n)$ denote the set of all commuting $n$-tuple of contractions on some $L^p$-space. The set of all homogeneous polynomials in $n$-variables of degree at most $k$ is denoted by $\mathbb{C}_k[z_1,\dots,z_n].$ Let $C_p(k,n)$ denote the quantity \[C_p(k,n)\colon=\sup\{\|P(\mathbf{T})\|_{L^p\to L^p}\colon \mathbf{T}\in \mathcal{C}_p(n),\ \|P(\mathcal{R})\|_{\ell^p(\mathbb{N}^n)\to\ell^p(\mathbb{N}^n)}\leq 1,\ P\in\mathbb{C}_k[z_1,\dots,z_n]\}\] and $C_p(n)\colon=\sup\limits_{k\in\mathbb N}C_p(k,n).$ Our main theorem is the following. Let $G$ denote the standard complex Gaussian random variable. The expectation of any random variable is denoted by $\mathbb E.$
	\begin{thm}\label{F} Let $1<p<2$ and $n\geq k\geq 3$. Then, there exists a constant $D\leq 8,$ such that
		\begin{align*}	C_p(k,n)\geq D^{-\frac{2}{p^\prime}}(1-o(1))\Big(\mathbb{E}|G|^{p^\prime}\Big)^{-\frac{k}{p^\prime}}n^{\frac{k-2}{p^\prime}}.
		\end{align*}
	\end{thm}Clearly, for $k\geq 3,$ $\lim_{n\to\infty}C_{p}(k,n)=\infty.$ Our approach goes back to \cite{VA74}, \cite{DI76} and the recent paper \cite{GAMS15}, where the authors used combinatorial and probabilistic techniques to study the asymptotic properties of constant $C_2(k,n).$ For $k=2,$ we prove the following theorem.
\begin{thm}\label{2decree} Let $1<p<2$ be such that $\frac{1}{2}9^{\frac{1}{p^\prime}}\Big(\mathbb{E}|G|^{p^\prime})^{\frac{1}{p^\prime}}\Big)^{-2}>1.$ Then, we have the estimate $\lim\limits_{k\to\infty}C_p(2,k(k-1))>1.$ 
	\end{thm} 
 In this case, we use some explicit examples which were constructed in \cite{FIR94} for studying Bell inequalities. We combine this examples with some of the techniques developed in \cite{GUR18} to obtain the required result which was used by the authors to solve a question of Varopoulos \cite{VA76}.

In Section \eqref{NCOM}, we discuss generalizations in the case of non-commutative $L^p$-spaces and semigroups. To state our results, we recall the basics of non-commutative $L^p$-spaces. For more on non-commutative $L^p$ spaces, we refer \cite{PIQ03}. Let $\mathcal M$ be a von Neumann algebra with normal, semifinite, faithful trace $\tau.$ Unless specified, we always work with von Neumann algebras of this type. Let us denote $\mathcal M_{+}$ to be the set of all positive elements in $\mathcal M.$ For any $x\in\mathcal M_{+},$ we denote the support projection of $x$ by $s(x).$ We denote $S(\mathcal{M})$ to be the linear span of set of all positive elements in $\mathcal M$ such that $\tau(s(x))<\infty$. For $1\leq p<\infty,$ we define the non-commutative $L^p$-space $L^p(\mathcal{M})$ to be the completion of $S({\mathcal{M}})$ with respect to the norm $\|x\|_{L^p(\mathcal{M})}\colon=\tau(|x|^p)^{\frac{1}{p}}$ where $|x|\colon=(x^*x)^{\frac{1}{2}}.$ One sets $L^{\infty}(\mathcal{M})=\mathcal{M}$. In particular for any Hilbert space $\mathcal{H}$ if we have $\mathcal{M}=B(\mathcal{H})$ and $\tau$ the usual trace denoted by $Tr$, then the corresponding non-commutative $L^p$-spaces are known to be Schatten-$p$ classes and denoted by $S^p(\mathcal{H})$, $1\leq p<\infty.$ Let $I$ be an indexing set. Then for $\mathcal{H}=\ell_I^2,$ the corresponding Schatten classes are denoted by $S^p_I$ for $1\leq p<\infty.$ We simply write $S^p_{\mathbb N}=S^p.$ For any two Hilbert spaces $\mathcal H_1$ and $\mathcal H_2,$ we denote $\mathcal H_1\otimes_2\mathcal H_2$ to be the usual Hilbert space tensor product. For any Banach space $X$ and $P\in\mathbb{C}[z_1,\dots,z_n],$ let us define
$\|P\|_{p,X}\colon=\|P(\mathcal{S})\otimes I_X\|_{\ell^p(\mathbb{N}^n,X)},$ where $\mathcal{S}$ is the commuting $n$-tuple of left shift operators on $\ell^p(\mathbb{N}^n).$
\begin{defn}[non-commutative $n$-Matsaev property]\label{NCMAT}Suppose $1<p<\infty$, $p\not=2$. A commuting $n$-tuple of contractions $\mathbf{T}=(T_1,\dots,T_n)$ on $L^p(\mathcal M)$ is
said to have non-commutative $n$-Matsaev property if 
$$\|P(\mathbf{T})\|_{L^p(\mathcal{M})\to L^p(\mathcal{M})}\leq\|P\|_{p,S^p}$$
for all polynomials $P\in\mathbb{C}[z_1,\dots,z_n].$
\end{defn}
 It is known that Akcoglu's dilation theorem does not have a non-commutative analogue \cite{JUL07}. This is a striking difference in the non-commutative universe. Also, till now it is not known if a multivariate analogue of Akcoglu's dilation theorem is true in the commutative setting. However, in \cite{AR13}, the author obtained a dilation theorem for a large class of unital completely positive Schur multipliers and Fourier multipliers on discrete group. In this paper, we show that a joint dilation theorem still holds for these class of operators. To state our results, we recall the definition of Schur multipliers. Let $M_I$ denote set of all matrices indexed by $I\times I.$
 \begin{defn} Let $1\leq p<\infty.$ A matrix $A\in M_I$ is called a Schur multiplier on $S^p_I$ if and only if the linear map $M_A\colon\text{Dom}(M_A)\subseteq B(\ell_I^2)\to B(\ell_I^2)$ defined as $M_A((b_{i,j})_{i,j\in I})\colon=(a_{i,j}b_{i,j})_{i,j\in I}$ extends to a bounded linear operator from
 $S_p^I$ to $S_p^I,$ where $1\leq p<\infty.$ We say that $A$ is a Schur multiplier on $B(\ell_I^2)$ if $M_A$ extends to a bounded operator from $B(\ell^2_I)$ to $B(\ell^2_I)$.
 \end{defn} We refer \cite{PA02} for more on Schur multipliers and the notion of completely positive maps. We prove following joint dilation theorem.
 \begin{thm}\label{NCD}
 	Let $M_{A_i}$ be unital completely positive Schur multipliers on $B(\ell_I^2)$ associated with real matrices $A_i$, $1\leq i\leq n.$ Then there exists a hyperfinite von Neumann algebra $\mathcal{M}$ equipped with a semifinite normal
 	faithful trace, a commuting tuple $(U_1,\dots,U_n)$ of unital trace preserving $*$-automorphisms on $\mathcal{M}$, a unital trace preserving one-to-one normal $*$-homomorphism $J\colon B(\ell_I^2)\to\mathcal{M}$ such that \[M_{A_1}^{k_1}\dots M_{A_n}^{k_n}=EU_1^{k_1}\dots U_n^{k_n}J,\] for all integers $k_i\geq 0,\ 1\leq i\leq n,$ where $E\colon\mathcal{M}\to B(\ell_I^2)$ is the canonical normal faithful trace preserving conditional expectation operator associated with $J.$
 \end{thm}
Let $1<p<\infty$. Then, the commuting tuple of unital completely positive Schur multipliers $(M_{A_1},\dots,M_{A_n})$ on $B(\ell_I^2)$ as described in \Cref{NCD} extends to a commuting tuple of contractive Schur multipliers $(M_{A_1},\dots,M_{A_n})$ on $S^p_I.$ We refer \cite[Section 4]{AR13} for detailed explanations.

Let $G$ be a discrete group. Let $g\mapsto\lambda(g)\colon\ell_G^2\to \ell_G^2$ be the left regular representation of $G$ and denote $VN(G)$ to be the group von Neumann algebra.
\begin{defn}[Fourier multiplier on $VN(G)$]A Fourier multiplier of $VN(G)$ is a normal linear map $M\colon VN(G)\to VN(G)$ such that there exists a function $t\colon G\to\mathbb{C}$ for which $M(\lambda(g))=t_g\lambda(g)$ for all $g\in G.$
\end{defn}
We denote the Fourier multiplier by $M_t.$ The group von Neumann algebra $VN(G)$ admits a normal semifinite faithful trace defined as $\tau(x)=\langle \epsilon_{e_G},x\epsilon_{e_G}\rangle_{\ell^2_G},$ where $(\epsilon_g)_{g\in G}$ is canonical basis of $\ell^2_G$ and $e_G$ denotes the identity element of $G.$
We prove the following joint dilation theorem for Fourier multipliers.
\begin{thm}\label{GVNA}Let $G$ be a discrete group and $M_{t_i}$ be unital completely positive Fourier multipliers on $VN(G)$, where $t_i\colon G\to\mathbb{C}$ are real functions for $1\leq i\leq n$. Then there exists a von Neumann algebra $\mathcal{M}$ equipped with a normal semifinite faithful trace, a commuting tuple of unital trace preserving $*$-automorphisms $(U_1,\dots,U_n)$ on $\mathcal{M}$ and a unital normal trace preserving one-one $*$-homomorphism $J\colon VN(G)\to\mathcal{M}$ such that $$M_{t_1}^{k_1}\dots M_{t_n}^{k_n}=EU_1^{k_1}\dots U_n^{k_n}J,$$ for all integers $k_i\geq 0,$ $1\leq i\leq n,$ where $E\colon \mathcal{M}\to VN(G)$ is the canonical faithful normal trace preserving conditional expectation operator associated with $J.$
\end{thm}
Let $1<p<\infty.$ Then, the commuting tuple of  unital completely positive Fourier multipliers $(M_{t_1},M_{t_2},\dots,M_{t_n})$ on $VN(G)$ described as in \Cref{GVNA}, extends to a commuting tuple of contractive Fourier multipliers $(M_{t_1},M_{t_2},\dots,M_{t_n})$ on $L^p(VN(G)).$ We refer \cite[Section 4]{AR13} for details.

By using \Cref{NCD} and \Cref{GVNA}, we have the following corollary.
\begin{cor}\label{coro}
	Let $1<p\neq 2<\infty.$ Then the following classes of commuting tuple of contractions satisfy non-commutative $n$-Matsaev property.
	\begin{itemize}
		\item[(i)] A commuting tuple of contractive Schur multipliers  $(M_{A_1},\dots,M_{A_n})$ on $S^p_I$ induced by a commuting tuple of unital completely positive Schur multipliers $(M_{A_1},\dots,M_{A_n})$ on $B(\ell^2_I),$ where $M_{A_i}:B(\ell^2_I)\to B(\ell^2_I)$ being associated with a real-valued matrix $A_i$ for $1\leq i\leq n.$
		\item[(ii)] A commuting tuple of contractive Fourier multipliers $(M_{t_1},M_{t_2},\dots,M_{t_n})$ on $L^p(VN(G))$ induced by a commuting tuple of unital completely positive Fourier multipliers $(M_{t_1},M_{t_2},\dots,M_{t_n})$ on $VN(G)$ where $M_{t_i}:VN(G)\to VN(G)$ being associated with a real-valued function $t_i:G\to\mathbb C$ and $G$ is a discrete amenable group or $G=\mathbb F_k$, i.e. free group with $k$-generators.
		\end{itemize}
\end{cor}
We skip the proof of above corollary which is along the same line as in \cite{AR13}.  
 We also establish a joint dilation theorem for multi-parameter semigroup of completely positive Schur multipliers, extending the dilation theorem  in \cite{AR13}, \cite{AR19} and \cite{AR191} in multivariate setting.

\section{$n$-Matsaev Property on commutative $L^p$ spaces}\label{COM} In this section, we generalize Matsaev's conjecture 
 in the multivariate setting and show that it fails to be true for some commuting tuple of contractions. We also show that the conjecture holds to be true for commuting tuple of isometries.

We recall notion of Fourier multipliers for locally compact abelian groups.
\begin{defn}[Fourier multipliers]\label{FM}
Let $G$ be a locally compact abelian group with its dual group $\widehat{G}$. For $1\leq p\leq\infty,$ a bounded operator $T\colon L^p(G)\to L^p(G)$ is called a Fourier multiplier, if there exists a $\phi\in L^\infty(\widehat{G}),$ such that  we have $\widehat{Tf}=\phi\widehat{f},$ for any $f\in L^2(G)\cap L^p(G).$
\end{defn}
We denote such an operator $T$ by $M_{\phi}$. We represent the set of all Fourier multipliers on $G$ by $M_p(G)$. For more on multiplier theory on locally compact abelian groups, we recommend the reader \cite{LA71}.

Let $1<p<\infty.$ Let $P\in\mathbb{C}[z_1,\dots,z_n]$ be a polynomial in $n$-variables.
Then, by the properties of multipliers and easy calculations, it would imply \[\|P(\mathcal{R})\|_{\ell^{p}(\mathbb{N}^n)\to \ell^{p}(\mathbb{N}^n)}=\|P(\mathcal{S})\|_{\ell^{p}(\mathbb{N}^n)\to \ell^{p}(\mathbb{N}^n)}=\|P\|_{M_p(\mathbb{Z}^n)}\] and \[\|P(\mathcal{R})\|_{\ell^{p^\prime}(\mathbb{N}^n)\to \ell^{p^\prime}(\mathbb{N}^n)}=\|P(\mathcal{S})\|_{\ell^{p^\prime}(\mathbb{N}^n)\to \ell^{p^\prime}(\mathbb{N}^n)}=\|P\|_{M_{p^\prime}(\mathbb{Z}^n)}=\|P\|_{M_p(\mathbb{Z}^n)},\] where $\frac{1}{p}+\frac{1}{p^\prime}=1.$

{\bf{Proof of Theorem \ref{T}:}}

\begin{proof} Let $R_j$ be the $j$-th right shift operator on $\ell^{p^\prime}(w,L^{p^\prime}(\Omega,\mathbb{F},\mu))$, $1\leq j\leq n.$ Suppose $F=(f_{i_1,\ldots,i_n})$ is an element of $\ell^{p^\prime}(w,L^{p^\prime}(\Omega,\mathbb{F},\mu)).$ By the properties of the weight $w$ (see  Part (2) of \ref{prope}), we have the following observation

\begin{eqnarray*}
\|R_jF\|^{p^\prime}_{\ell^{p^\prime}(w,L^{p^\prime}(\Omega,\mathbb{F},\mu))}
\leq M_j\|F\|^{p^\prime}_{\ell^{p^\prime}(w,L^{p^\prime}(\Omega,\mathbb{F},\mu))}.
\end{eqnarray*}
Thus $R_j$ becomes a bounded operator for each $j,\ 1\leq j\leq n$. We observe that as Banach spaces $(\ell^{p^\prime}(w,L^{p^\prime}(\Omega,\mathbb{F},\mu)))^*\cong \ell^p(v,L^p(\Omega,\mathbb{F},\mu))$, where $v=(v_{i_1,\dots,i_n})$ is a sequence of positive real numbers indexed by $\mathbb{N}^n$ and for each $i_j\in\mathbb{N},1\leq j\leq n,$ $v_{i_1,\dots,i_n}={w_{i_1,\dots,i_n}}^{-\frac{p}{p^\prime}},$ where $\frac{1}{p}+\frac{1}{p^\prime}=1.$ We have the duality relationship between the Banach spaces, which is given by $\langle F,G\rangle\colon=\sum_{i_j}\int_\Omega f_{i_1,\dots,i_n}g_{i_1,\dots,i_n},$ where $F=(f_{i_1,\dots,i_n})$ is an element of $\ell^p(v,L^p(\Omega,\mathbb{F},\mu))$ and $G=(g_{i_1,\dots,i_n})$ is in $\ell^{p^\prime}(w,L^{p^\prime}(\Omega,\mathbb{F},\mu))$. Thereafter for each $j,$ $1\leq j\leq n$, let us consider the adjoint operator \[R_j^*\colon\ell^p(v,L^p(\Omega,\mathbb{F},\mu))\to \ell^p(v,L^p(\Omega,\mathbb{F},\mu)).\] It is not hard to check that 
\begin{equation}\label{o1}
R_j^*(F)=(f_{i_1},\dots,f_{i_j+1},\dots,f_{i_n}),
\end{equation}
 where $F=(f_{i_1,\dots,i_n})$ is in $\ell^p(v,L^p(\Omega,\mathbb{F},\mu))$. We choose a positive real number $\lambda$ in $(0,1)$ and define $V_i=\lambda U_i,$ $1\leq i\leq n$. Now if we consider the operator $W\colon L^p(\Omega,\mathbb{F},\mu)\to \ell^p(v,L^p(\Omega,\mathbb{F},\mu)),$ 
 \begin{equation}\label{o2}
 Wf=(V_1^{i_1-1}\cdots V_n^{i_n-1}f),
 \end{equation} then one can see that
\begin{equation}\label{kkb}
\|Wf\|^p_{\ell^p(v,L^p(\Omega,\mathbb{F},\mu))}
=M\|f\|^p_{L^p(\Omega,\mathbb{F},\mu)}
\end{equation}
where $M\colon=\sum_{i_j}{w_{i_1,\dots,i_n}}^{-\frac{p^\prime}{p}}{\lambda}^{p((i_1+\cdots+i_n)-n)}.$ Since we have $w_{i_1,\dots,i_n}\geq1$ (see Part (1) of \ref{prope}) for all $i_j\in\mathbb{N},1\leq j\leq n$ and $\lambda$ is in $(0,1),$ the quantity $M$ is finite. We notice that, for any tuple $(m_1,\dots,m_n)$ of non-negative integers, the identity \begin{equation}\label{o3}
{R_1^*}^{m_1}\cdots{R_n^*}^{m_n}Wf=(V_1^{i_1+m_1-1}\cdots V_1^{i_n+m_n-1}f)
\end{equation} holds by using equation \ref{o1}. Also one can observe that 
\begin{equation}\label{o4}
WV_1^{m_1}\cdots V_n^{m_n}f=(V_1^{i_1+m_1-1}\cdots V_1^{i_n+m_n-1}f).
\end{equation} It is worth noticing that here, we actually use the commuting properties of the operators $U_i$'s. Therefore, one obtains by equations \ref{o3} and \ref{o4} that
\[{R_1^*}^{m_1}\cdots {R_n^*}^{m_n}Wf=WV_1^{m_1}\cdots V_n^{m_n}f.\] We readily see that for any polynomial $P$ in $n$-variables \[P(\mathcal{R^*})Wf=WP(\mathbf{V})f,\] where $\mathcal{R^*}\colon=({R_1}^*,\dots,{R_n}^*)$ and $\mathbf{V}\colon=(V_1,\dots,V_n).$ We make use of this to obtain the following inequality
\begin{equation}\label{b}
\|WP(\mathbf{V})f\|_{\ell^p(v,L^p(\Omega,\mathbb{F},\mu))}\leq \|P(\mathcal{R^*})\|_{\ell^p(v,L^p(\Omega,\mathbb{F},\mu))\to\ell^p(v,L^p(\Omega,\mathbb{F},\mu))}\|Wf\|_{\ell^p(v,L^p(\Omega,\mathbb{F},
	\mu))}.
\end{equation}
Henceforth, following Equations \ref{kkb} and \ref{b}, we observe that for any polynomial $P$ in $n$-variables, one has the inequality \[\|P(\mathbf{V})\|_{L^p(\Omega,\mathbb{F},\mu)\to L^p(\Omega,\mathbb{F},\mu)}\leq\|P(\mathcal{R^*})\|_{\ell^p(v,L^p
	(\Omega,\mathbb{F},\mu))\to \ell^p(v,L^p(\Omega,\mathbb{F},\mu))}.\] In above inequality, we take $\lambda$ arbitrarily close to one to obtain \begin{equation}\label{o5}\|P(\mathbf{U})\|_{L^p(\Omega,\mathbb{F},\mu)\to L^p(\Omega,\mathbb{F},\mu)}\leq\|P(\mathcal{R^*})\|_{\ell^p(v,L^p(\Omega,\mathbb{F},\mu))\to\ell^p(v,L^p(\Omega,\mathbb{F},\mu))}.
\end{equation} Next we show that $\|P(\mathcal{R^*})\|_{\ell^p(v,L^p(\Omega,\mathbb{F},\mu))\to\ell^p(v,L^p(\Omega,\mathbb{F},\mu))}=\|P(\mathcal{R})\|_{{\ell^{p}(w,\mathbb{N}^n)}\to{\ell^{p}(w,\mathbb{N}^n)}}.$ Hence we choose $F=(f_{i_1,\dots,i_n})$ in $\ell^{p^\prime}(w,L^{p^\prime}(\Omega,\mathbb{F},\mu))$ and $G=(g_{i_1,\dots,i_n}),$ where $P(\mathcal{R})F=G$. We have the following observation
\begin{eqnarray*}
\|P(\mathcal{R})F\|^{p^\prime}_{\ell^{p^\prime}(w,L^{p^\prime}(\Omega,\mathbb{F},\mu))}\leq\|P(\mathcal{R})\|^{p^\prime}_{{\ell^{p^\prime}(w,\mathbb{N}^n)}\to {\ell^{p^\prime}(w,\mathbb{N}^n)}}\|F\|^{p^\prime}_{\ell^{p^\prime}(w,L^{p^\prime}(\Omega,\mathbb{F},\mu))}.
\end{eqnarray*}

Again let us consider an element $h$ in $L^{p^\prime}(\Omega,\mathbb{F},\mu))$ with $\|h\|_{L^{p^\prime}(\Omega,\mathbb{F},\mu))}$ to be one. Suppose $C=(c_{i_1,\dots,i_n})$ is a sequence indexed by $\mathbb{N}^n$ and define $h_{i_1,\dots,i_n}=c_{i_1,\dots,i_n}h$. Hence if we define $H=(h_{i_1,\dots,i_n})$, we have $\|H\|_{\ell^{p^\prime}(w,L^{p^\prime}(\Omega,\mathbb{F},\mu))}=
\|C\|_{{\ell^{p^\prime}(w,\mathbb{N}^n)}}$ and $\|P(\mathcal{R})H\|^{p^\prime}_{\ell^{p^\prime}
	(w,L^{p^\prime}(\Omega,\mathbb{F},\mu))}=\|P(\mathcal{R})C\|^{p^\prime}_
{\ell^{p^\prime}(w,\mathbb{N}^n)}.$ Therefore we obtain the inequality \[\|P(\mathcal{R})\|^{p^\prime}_{\ell^{p^\prime}(w,L^{p^\prime}
	(\Omega,\mathbb{F},\mu))\to\ell^{p^\prime}(w,L^{p^\prime}(\Omega,\mathbb{F}
	,\mu))}\geq\|P(\mathcal{R})\|^{p^\prime}_{{\ell^{p^\prime}(w,\mathbb{N}^n)}\to{\ell^{p^\prime}(w,\mathbb{N}^n)}}.\]  Thus, we have from above  
\[\|P(\mathcal{R})\|_{\ell^{p^\prime}(w,L^{p^\prime}(\Omega,\mathbb{F},\mu))\to
	\ell^{p^\prime}(w,L^{p^\prime}(\Omega,\mathbb{F},\mu))}=\|P(\mathcal{R})\|_{{\ell^{p^\prime}(w,\mathbb{N}^n)}\to{\ell^{p^\prime}(w,\mathbb{N}^n)}}.\] Using  duality, one can immediately notice that \[\|P(\mathcal{R})\|_{{\ell^{p^\prime}(w,\mathbb{N}^n)}\to{\ell^{p^\prime}
		(w,\mathbb{N}^n)}}=\|P(\mathcal{R^*})\|_{{\ell^{p}(w,\mathbb{N}^n)}\to{\ell^{p}(w,\mathbb{N}^n)}}.$$ Also, it is not hard to observe that $$\|P(\mathcal{R})\|_{{\ell^{p}(w,\mathbb{N}^n)}\to{\ell^{p}(w,\mathbb{N}^n)}}=\|P(\mathcal{S})\|_{{\ell^{p}(w,\mathbb{N}^n)}\to{\ell^{p}(w,\mathbb{N}^n)}}.\]
In view of \ref{o5}, the proof of the theorem is completed.
\end{proof}

\begin{rem}From the above theorem, it trivially follows that any commuting $n$-tuple of isometries (not necessarily onto) has $n$-Matsaev property. 
\end{rem}

 Motivated by \cite{AR13}, we have the following generalization for general Banach spaces. The proof follows as in \cite{AR13} with necessary modifications as done in the previous theorem.
\begin{thm}\label{GMAT}Let $1\leq p\leq\infty$ and $X$ be a Banach space. Suppose $\mathbf{T}=(T_1,\dots,T_n)$ is a commuting tuple of isometries on $X.$ Then, for any polynomial $P$ in $n$-variables,
we have the following inequality
$\|P(\mathbf{T})\|_{X\to X}\leq\|P\|_{p,X}.$
\end{thm}

\subsection{Failure of the multivariate Matsaev's conjecture}\label{FAIL}In this subsection, we prove that the multivariate analogue of Matsaev's conjecture fails for all $1<p<\infty.$ Our construction uses explicit examples which were constructed in \cite{GAMS15}.  For any set $\mathbf{I},$ denote $\#\mathbf{I}$ to be the cardinality of $\mathbf{I}.$ We need the following lemma.

\begin{lem}\cite{PI89}\label{PIEM} Let $1<p<\infty.$ Then, for every $n\geq 1,$ there exists a subspace $E_n$ of $L^p([0,1])$ such that $E_n$ is isometrically isomorphic to $\ell^2_n$ and the projection map $\mathbb{P}_n\colon L^p([0,1])\to E_n\subseteq L^p([0,1])$ has the property that
\begin{eqnarray}\label{Se}
\|\mathbb{P}_n\|_{L^p([0,1])\to L^p([0,1])}\leq
\begin{cases}(\mathbb{E}|G|^p)^{\frac{1}{p}} ,& 2\leq p<\infty\\
        (\mathbb{E}|G|^{p^\prime})^{\frac{1}{p^\prime}},& 1<p<2.
         \end{cases}
\end{eqnarray}
\end{lem}
We would require the notion of partial Steiner system which were used in \cite{GAMS15} for constructing commuting tuple of contractions on Hilbert spaces. For more on partial Steiner systems and the properties of them, which we use here, we refer \cite{GAMS15} and references therein.
\begin{defn}
	 Let $n\in \mathbb N$ and $1\leq t\leq k\leq n.$ A partial
	Steiner system with the parameters $t$, $k$ and $n$ is a collection of subsets of $\{1,2,\dots,n\}$ of
	cardinality $k$ called as blocks which has the property that every subset of $\{1,2,\dots,n\}$ of
	size $t$ is contained in at most one block of the system. We denote a partial Steiner system
	by $S_p(t,k,n).$
\end{defn}
{\bf{Proof of Theorem \ref{F}} :}

\begin{proof} We fix $n\geq k\geq 3,n\in\mathbb{N}$. As in \cite[Proof of Theorem 1.1-(i)]{GAMS15}, we take commuting tuple of contractions $(T_1,\dots,T_n)$ acting on a finite dimensional Hilbert space $\mathcal{H}$. Consider the homogeneous polynomial of degree $k$ in $n$-variables, \[P(z_1,\dots,z_n)=\sum_{\{i_1,\dots,i_k\}\in S_p(k-1,k,n)}c_{i_1,\dots,i_k}z_{i_1}\dots z_{i_k}\] with $c_{i_1,\dots,i_k}\in\{1,-1\}$ and $S_p(k-1,k,n)$ a Partial Steiner system as in \cite[Proof of Theorem 1.1-(i)]{GAMS15}. Together with elementary results about Fourier multipliers (see \cite{LA71}), $\|P\|_{M_2(\mathbb Z^n)}=\|P\|_{\infty,\mathbb T^n}$ and $\|P\|_{M_1(\mathbb Z^n)}=\sum_{\{i_1,\dots,i_k\}\in S_p(k-1,k,n)}|c_{i_1,\dots,i_k}|$, we have following facts from \cite{GAMS15}.
	  \begin{itemize} 
	  	\item[(i)] $\|P(T_1,\dots,T_n)\|_{\mathcal H\to\mathcal H}\geq \# S_p(k-1,k,n).$
	  	\item[(ii)] $\|P\|_{M_2(\mathbb Z^n)}\leq D\Big(n\log k\# S_p(k-1,k,n)\Big)^{\frac{1}{2}}$ with $D$ being an absolute constant less than $8.$
	  	\item[(iii)] $\|P\|_{M_1(\mathbb Z^n)}=\# S_p(k-1,k,n).$
	  	\item[(iv)]  $\# S_p(k-1,k,n)\geq \frac{\binom n{k-1}}{k}(1-o(1)).$
	  \end{itemize}
 We have the following estimate on multiplier norm by complex interpolation \cite{LA71},
 \begin{equation}\label{sugge1}
 \|P\|_{M_p(\mathbb Z^n)}\leq {\# S_p(k-1,k,n)}^{-1+\frac{2}{p}}{\Big(D\sqrt{n\log k\# S_p(k-1,k,n)}}\Big)^{2(1-\frac{1}{p})}.
 \end{equation}
	  
  	 By Lemma \ref{PIEM}, we can view $\mathcal{H}$ as a complemented subspace of $L^p([0,1]).$ Slightly abusing the notations, we write \[L^p([0,1])=\mathcal H\oplus F.\] For $1\leq i\leq n,$ define $\widetilde{T_i}:L^p([0,1])\to L^p([0,1])$ as $\widetilde{T_i}(x\oplus y)\colon=T(x)\oplus 0,x\in\mathcal H,y\in F.$ It follows from \Cref{PIEM} that we have $\|\widetilde{T_i}\|_{L^p([0,1])\to L^p([0,1])}\leq(\mathbb{E}|G|^{p^\prime})^{\frac{1}{p^\prime}}.$ Note that $\mathbf{\widetilde{T}}\colon=(\widetilde{T_1},\dots,\widetilde{T_n})$ is a commuting tuple. Therefore, if we define $S_i\colon=(\mathbb{E}|G|^{p^\prime})^{-\frac{1}{p^\prime}}
  	 \widetilde{T_i}$, then $\mathbf S\colon=(S_1,\dots,S_n)$ is a commuting tuple of contractions on $L^p([0,1]).$ 

Let us notice the following 
\begin{eqnarray}\label{sugge}\nonumber\|P(\mathbf S)\|_{L^p([0,1])\to L^p([0,1])}&=&\Big\|\sum_{\{i_1,\dots,i_k\}\in S_p(k-1,k,n)}c_{i_1,\dots,i_k}S_{i_1}\dots S_{i_k}\Big\|_{L^p([0,1])\to L^p([0,1])}\\ \nonumber
&
=&
\Big\|\Big(\mathbb{E}|G|^{p^\prime}\Big)^{-\frac{k}{p^\prime}}\sum_{\{i_1,\dots,i_k\}
	\in S_p(k-1,k,n)}c_{i_1,\dots,i_k}\widetilde{T_{i_1}}\widetilde{T_{i_1}}\dots\widetilde{T_{i_k}}
\Big\|_{L^p([0,1])\to L^p([0,1])}\\\nonumber
&=&\Big(\mathbb{E}|G|^{p^\prime}\Big)^{-\frac{k}{p^\prime}}\sup_{z\in L^p([0,1]),\|z\|_p=1}\Big\|P(\widetilde{\mathbf T})z\Big\|_{L^p([0,1])}\\\nonumber
&\geq& \Big(\mathbb{E}|G|^{p^\prime}\Big)^{-\frac{k}{p^\prime}}\sup_{x\in \mathcal H,\|x\|_{\mathcal H}=1}\Big\|P(\widetilde{\mathbf T})(x\oplus 0)\Big\|_{\mathcal H}\\\nonumber
&=&\Big(\mathbb{E}|G|^{p^\prime}\Big)^{-\frac{k}{p^\prime}}\sup_{x\in \mathcal H,\|x\|_{\mathcal H}=1}\Big\|P(\mathbf{T})x\Big\|_{\mathcal H}\\
&=&\Big(\mathbb{E}|G|^{p^\prime}\Big)^{-\frac{k}{p^\prime}}\|P(\mathbf T)\|_{\mathcal H\to\mathcal H}\geq\Big(\mathbb{E}|G|^{p^\prime}\Big)^{-\frac{k}{p^\prime}}\# S_p(k-1,k,n).
\end{eqnarray}
Therefore, by \eqref{sugge1} and \eqref{sugge}, we have the following estimate.
\begin{eqnarray}\label{rkk}
\nonumber\frac{\|P(\mathbf{S})\|_{L^p([0,1])\to L^p([0,1])}}{\|P\|_{M_p(\mathbb{Z}^n)}}&\geq &\frac{\Big(\mathbb{E}|G|^{p^\prime}\Big)^{-\frac{k}{p^\prime}}\# S_p(k-1,k,n)}{{\# S_p(k-1,k,n)}^{-1+\frac{2}{p}}{(D\sqrt{n\log k\# S_p(k-1,k,n)}})^{2(1-\frac{1}{p})}}\\
&=&\frac{{\Big(\mathbb{E}|G|^{p^\prime}\Big)^{-\frac{k}{p^\prime}}\# S_p(k-1,k,n)}^{\frac{1}{p^\prime}}}{(D\sqrt{n\log k})^{\frac{2}{p^\prime}}}.
\end{eqnarray}
We use the estimate $\# S_p(k-1,k,n)\geq \frac{\binom n{k-1}}{k}(1-o(1))$ and obtain from \eqref{rkk}
\begin{eqnarray*}C_p(k,n)&\geq &\frac{\|P(\mathbf{S})\|_{L^p([0,1])\to L^p([0,1])}}{\|P\|_{M_p(\mathbb{Z}^n)}}\\
	&\geq &\Big(\mathbb{E}|G|^{p^\prime}\Big)^{-\frac{k}{p^\prime}}
(\frac{1}{k}{{n}\choose{k-1}}(1-o(1)))^{\frac{1}{p^\prime}}(D\sqrt{n\log k})^{-\frac{2}{p^\prime}}\\
&\geq &D^{-\frac{2}{p^\prime}}(1-o(1))\Big(\mathbb{E}|G|^{p^\prime}\Big)^{-\frac{k}{p^\prime}}
\Big(\frac{{{n}\choose{k-1}}}{nk\log k}\Big)^{\frac{1}{p^\prime}}\\
&\geq &D^{-\frac{2}{p^\prime}}(1-o(1))\Big(\mathbb{E}|G|^{p^\prime}\Big)^{-\frac{k}{p^\prime}}n^{\frac{k-2}{p^\prime}}.
\end{eqnarray*}
This completes the proof of the theorem.
\end{proof}
For any $n\times n$ complex matrix $A\colon=\big (\! \big (a_{ij}\big )\!\big ),$
define $P_{\!_A},$ as $P_{\!_A}(z_1,\ldots,z_n)=\sum_{i,j=1}^n a_{ij} z_i z_j.$ Let $\mathcal{H}$ be a separable Hilbert space and
$\{e_j\}_{j\in\mathbb{N}}$ be a orthonormal basis of $\mathcal{H}$.
For any $x\in\mathcal{H}$,
let us define $x^{\sharp}\colon\mathcal{H}\to\mathbb C$ by $x^{\sharp}(y)=\sum_{j}x_jy_j,$
where $x=\sum_j x_je_j$ and $y=\sum_j y_je_j$.
For $x,y\in\mathcal{H}$, we set $\left[x^{\sharp},y\right]=x^{\sharp}(y).$
Let $\mathcal{H}^{\sharp}\colon=\left\{x^{\sharp}\colon x\in\mathcal H\right\}.$
Let $\mathcal{H}^{\sharp}$ be equipped with the operator norm.
The map $\phi\colon\mathcal{H}\to\mathcal{H}^{\sharp}$ defined by
$\phi(x)=x^{\sharp}$ is a linear onto isometry,
where $\mathcal{H}^{\sharp}$ is equipped with the operator norm.
\begin{defn}[Varopoulos Operator]\cite{GU15}
	Let $\mathcal{H}$ be a separable Hilbert space.
	For $x,y\in\mathcal{H}$,
	define $T_{x,y}\colon\mathbb C \oplus \mathcal{H} \oplus \mathbb C \to \mathbb C \oplus \mathcal{H} \oplus \mathbb C$ by
	\[T_{x,y}=
	\left(
	\begin{array}{ccc}
		0 & x^{\sharp} & 0\\
		0 & 0 & y\\
		0 & 0 & 0\\
	\end{array}
	\right).\]
	The operator $T_{x,y}$ will be called Varopoulos operator
	corresponding to the pair of vectors $x,y$.
	If $x=y$ then $T_{x,y}$ will simply be denoted by $T_x$.
\end{defn}
\begin{lem}\label{sur}\cite{GUR18}
Let $(a_{ij})_{i,j=1}^n$ be a non-negative definite matrix. Then,
\[\sup\limits_{|z_i|=|z_j|=1,1\leq i,j\leq n}\Bigg|\sum_{i,j=1}^na_{ij}z_iz_j\Bigg|=\sup_{z_i,z_j\in\{+1,-1\},1\leq i,j\leq n}\sum_{i,j=1}^na_{ij}z_iz_j.\]
\end{lem}

{\bf {Proof of Theorem \ref{2decree}}:}

\begin{proof}

To prove the theorem, we need the Reeds-Fishburn matrices which were used in the beginning of Section 3. in \cite{GUR18} to produce a large class of concrete examples for which von Neumann inequality fails. This example originally goes back to \cite{FIR94}. For $l=k(k-1)$, define $F_k=\{v_1,\dots,v_{k(k-1)}\}$,
the set of all $k$-dimensional vectors with two non-zero components,
either $1$ and $1$ or $1$ and $-1,$ appearing in that order.
Define a real $l\times l$ non-negative definite matrix
$A_k=(a_{ij})_{1\leq i,j\leq l}$
where $a_{ij}=\langle v_i,v_j\rangle$ for $1\leq i,j\leq l.$ It has been proved in \cite{FIR94} that
$$\sup_{\|X_i\|_{\ell_\mathbb{R}^2}=1}\sum_{i,j=1}^la_{ij}\langle X_i,X_j\rangle=2k(k-1)^2\ \text{and}\sup_{\omega_i\in\{1,-1\}}\sum_{i,j=1}^la_{ij}
\omega_i\omega_j=\frac{2k(k-1)(2k-1)}{3}.$$ A careful counting argument gives us $\sum_{i,j=1}^l|a_{i,j}|=2k(k-1)(2k-3).$ Let us consider the polynomial $P_{A_k}(z_1,\dots,z_l)=\sum_{i,j=1}^la_{ij}z_iz_j$ to be the Reeds-Fishburn polynomial of order $l$, where $(a_{ij})_{i,j=1}^l$ is the Reeds-Fishburn matrix of order $l.$ By Lemma \ref{sur} and Riesz-Thorin complex interpolation, we obtain $$\|P_{A_k}\|_{M_p(\mathbb{Z}^l)}\leq(2k(k-1)(2k-3))^{1-\frac{2}{p^\prime}}\big(\frac{2k(k-1)(2k-1)}{3}\big)^{\frac{2}{p^\prime}}.$$ Therefore, for this class of polynomials, we have
\begin{equation}\label{darkar}
\frac{\sup_{\|X_i\|_{\ell_\mathbb{R}^2}=1}\sum_{i,j=1}^la_{ij}\langle X_i,X_j\rangle}{\|P_{A_k}\|_{M_p(\mathbb{Z}^l)}}\geq\frac{2k(k-1)^2}{(2k(k-1)(2k-3))^{1-\frac{2}{p^\prime}}\big(\frac{2k(k-1)(2k-1)}{3}\big)^{\frac{2}{p^\prime}}}.
\end{equation} If we consider, the commuting tuple of Varopoulos operators $\mathbf{T_X}=\{T_{X_i}\colon 1\leq i\leq l\},$ where $X_i$'s are real $\ell^2$-unit vectors (see \cite{GUR18} for more explanation). We can immediately see that $$\sup_{\|X_i\|_{\ell_\mathbb{R}^2}=1}\|P_{A_k}(\mathbf{T_X})\|_{\ell_\mathbb{R}^2\to\ell_\mathbb{R}^2 }=\sup_{\|X_i\|_{\ell_\mathbb{R}^2}=1}\sum_{i,j=1}^la_{ij}\langle X_i,X_j\rangle.$$ Therefore, by \eqref{darkar} and above, we have an estimate $$\frac{\|P_{A_k}(\mathbf{T_X})\|_{\ell_{\mathbb R}^2\to\ell_{\mathbb R}^2}}{\|P_{A_k}\|_{M_p(\mathbb{Z}^l)}}\geq\frac{2k(k-1)^2}{(2k(k-1)(2k-3))^{1-\frac{2}{p^\prime}}\big(\frac{2k(k-1)(2k-1)}{3}\big)^{\frac{2}{p^\prime}}},$$ which clearly goes to $\frac{1}{2}9^{p^\prime}$ as $k\to\infty.$ Therefore, proceeding as in Theorem \ref{F} and pretending that $T_{X_i}$'s are acting on the Hilbert space part of $L^p([0,1])$, we have by Lemma \ref{PIEM} that \[\lim_{k\to\infty}C_p(2,k(k-1))\geq \frac{1}{2}9^{\frac{1}{p^\prime}}\Big(\mathbb{E}|G|^{p^\prime})^{\frac{1}{p^\prime}}\Big)^{-2}>1.\]
This completes the proof of the theorem.
\end{proof}
\begin{rem}Note that $\lim_{p\to 2^-}\Big(\mathbb{E}|G|^{p^\prime})^{\frac{1}{p^\prime}}\Big)^{-2}=1.$ Therefore, there is a range of $p$ for which the above theorem holds true.
\end{rem}

One can produce a family of examples for which $C_p(2,3)>1$ for a range of $p\in(1,2)$ by deploying the examples produced in \cite{GUR18} in our situation.

\[C_p(2,3)\geq
\begin{cases}
\Big(\mathbb{E}|G|^{p^\prime})^{\frac{1}{p^\prime}}\Big)^{-2}\frac{3-2\alpha}{(7+2\alpha)^{\frac{2}{p^\prime}}(6+2|\alpha|)^{1-\frac{2}{p^\prime}}}, & \alpha \in (-1/2,0)\\
\Big(\mathbb{E}|G|^{p^\prime})^{\frac{1}{p^\prime}}\Big)^{-2}\frac{3-2\alpha-1/\alpha}{(7+2\alpha)^{\frac{2}{p^\prime}}(6+2|\alpha|)^{1-\frac{2}{p^\prime}}}, & \alpha \in [-1,-1/2]\\
\Big(\mathbb{E}|G|^{p^\prime})^{\frac{1}{p^\prime}}\Big)^{-2}\frac{3-2\alpha-1/\alpha}{(3-2\alpha)^{\frac{2}{p^\prime}}(6+2|\alpha|)^{1-\frac{2}{p^\prime}}}, & \alpha \in (-\infty,-1).
\end{cases}
\]
\section{Generalization on non-commutative $L^p$ spaces }\label{NCOM}
In this section, we consider multivariate Matsaev's conjecture in the set up of non-commutative $L^p$-spaces. Our present section is motivated by \cite{AR13}. We follow notations and terminologies from \cite{AR13}. However, for more elaboration on the various notions we use in this section, we refer \cite{AR13} and references therein. Let $\mathcal{M}$ be a von Neumann algebra with a normal semifinite faithful trace $\tau$. Then, denote $B(\ell_I^2)\overline{\otimes}\mathcal{M}$ to be the von Neumann algebra tensor product. It is clear that $Tr\otimes\tau$
becomes a normal semifinite faithful trace on $B(\ell_I^2)\overline{\otimes}\mathcal{M}.$ Also the algebraic tensor product $S^p_I\otimes L^p(\mathcal{M})$ is dense in the associated non-commutative $L^p$-space,
$L^p(B(\ell_I^2)\overline{\otimes}\mathcal{M},Tr\otimes\tau)$ for $1\leq p<\infty.$ To prove the next theorem, we recall the basics of Fermion algebras.

Let $S_n$ be the permutation group over the set $\{1,\dots,n\}.$ If $\sigma\in S_n,$ we denote $|\sigma|$ to be the total number of inversions of $\sigma,$ i.e.,
$|\sigma|\colon=\#\{(i,j)\colon\sigma(i)>\sigma(j),1\leq i,j\leq n\}.$ Let $H$ be a real Hilbert space and $H_{\mathbb{C}},$ its complexification. Consider the vector space
$$\mathbb{C}v\oplus\bigoplus_{n\geq 1}H_{\mathbb{C}}^{\otimes n},$$ with a sesquilinear form defined as
\[\langle h_1\otimes\cdots\otimes h_n,k_1\otimes\cdots k_n\rangle_{-1}\colon=\sum_{\sigma\in S_n}{(-1)}^{|\sigma|}\langle h_1,k_{\sigma(1)}\rangle_{H_{\mathbb{C}}}\cdots\langle h_n,k_{\sigma(n)}\rangle_{H_{\mathbb{C}}},\] where the unit vector $v$ is called the {\it vacuum}. One needs to take quotient of the space $\mathbb{C}v\oplus\bigoplus_{n\geq 1}H_{\mathbb{C}}^{\otimes n}$ by the kernel of $\langle .,. \rangle_{-1}$ to get an inner product. In this way, one obtains  antisymmetric Fock space denoted by $\mathcal F_{-1}(H)$. For $e\in H,$ one defines the so called
creation operator  $l(e)\colon\mathcal{F}_{-1}(H)\to \mathcal{F}_{-1}(H)$ as $$l(e)(h_1\otimes\cdots\otimes h_n)\colon=e\otimes h_1\otimes\cdots\otimes h_n.$$ The creation operators satisfy the following relation
$$l(f)^*l(e)+l(e)l(f)^*=\langle f,e\rangle_H I_{\mathcal{F}_{-1}(H)}.$$ We denote $w(e)\colon\mathcal{F}_{-1}(H)\to \mathcal{F}_{-1}(H)$ to be the operator \[w(e)=l(e)+l(e)^*.\] The von Neumann algebra $\Gamma_{-1}(H)$ is the von Neumann algebra generated by
$\{w(e)\colon e\in H\}.$ It is a finite von Neumann algebra with the trace $\tau(x)=\langle v,xv\rangle,$ where $x\in\Gamma_{-1}(H).$ The space $\Gamma_{-1}(H)$ is called the Fermion algebra. It follows from the Wick formula that \[\tau(\omega(e)\omega(f))=\langle e,f\rangle_{H}.\] Also $\omega(e)^2=I_{\mathcal{F}_{-1}(H)},$ $\|e\|_H=1.$ We refer \cite{BOS91} and \cite{BOKS97} for detailed discussions on this topic. 

For simplicity, we give the proof of Theorem \ref{NCD} for two commuting unital completely positive Schur multipliers $(M_A,M_B)$ associated with real matrices $A$ and $B$ respectively.

{\bf{Proof of Theorem \ref{NCD}}}
\begin{proof}
Denote $A=(a_{ij})_{ij\in I}$ and $B=(b_{ij})_{i,j\in I}.$ Since the operators $M_A$ and $M_B$ are completely positive on $B(\ell_I^2),$ one can define positive symmetric bilinear forms $\langle.,.\rangle_A$ and $\langle.,.\rangle_B$ on the real linear span of $e_i,i\in I$ by
$\langle e_i,e_j\rangle_A=a_{ij}$ and $\langle e_i,e_j\rangle_B=b_{ij}$ respectively. We denote by $\mathcal{H}_A$ and $\mathcal{H}_B$ to be the completion of the real pre-Hilbert space obtained by taking quotient by the corresponding
kernels of $\langle.,.\rangle_A$ and $\langle.,.\rangle_B$ respectively. Let us denote $[e_i]_A$ and $[e_i]_B$ to be the equivalence classes corresponding to $e_i$ in $\mathcal{H}_A$ and $\mathcal{H}_B$ respectively for $i\in I.$ To distinguish between $\mathcal H_A$ and $\mathcal H_B$ and objects associated to them, we simply add suffix by $A$ and $B$ respectively to the introduced notations. We consider the following tensor von Neumann algebra \[\mathcal{M}=B(\ell_I^2)\overline{\otimes}\big(\overline{\otimes}_
{\mathbb{Z}}((\Gamma_{-1}(\mathcal{H}_A),\tau_A)\overline{\otimes}(\Gamma_{-1}(\mathcal{H}_B),\tau_B))\big)\]equipped with faithful semifinite normal tensor trace $\tau_{\mathcal{M}}=Tr\otimes\cdots\otimes(\tau_A\otimes\tau_B)
\otimes\cdots.$ We define the following element
\[d_1\colon=\sum_{i\in I}e_{ii}\otimes\cdots\underbrace{\otimes}_0(w_A([e_i]_A)\otimes I)\otimes(I\otimes I)\otimes\cdots,\] where $w([e_i]_A)\otimes I$ is in the $0$-th position. In a similar fashion define  \[d_2\colon=\sum_{j\in I}e_{jj}\otimes\cdots\underbrace{\otimes}_0(I\otimes w_B([e_j]_B))\otimes(I\otimes I)\otimes\cdots.\] As in the singe variable case \cite{AR13}, one easily checks that $d_1$ and $d_2$ are symmetries, i.e., self-adjoint unitary elements. We have the normal faithful trace preserving conditional expectation
operator $E\colon\mathcal{M}\to B(\ell_I^2)$ as \[E=I_{B(\ell_I^2)}\otimes\cdots\otimes(\tau_A\otimes\tau_B)\otimes\cdots.\]
Again define the inclusion map $J\colon B(\ell_I^2)\to\mathcal{M}$ as following  \[J(x)\colon=x\otimes\cdots\otimes(I\otimes I)\otimes\cdots.\] Note that $J$ is injective normal unital $*$-homomorphism which preserves trace.
For $i=1,2,$ we define the following shift operators
\[\mathcal{S}_i\colon\overline{\otimes}_{\mathbb{Z}}((\Gamma_{-1}(\mathcal{H}_A)
,\tau_A)\overline{\otimes}(\Gamma_{-1}
(\mathcal{H}_B),\tau_B))\to \overline{\otimes}_{\mathbb{Z}}((\Gamma_{-1}(\mathcal{H}_A),\tau_A)\overline{\otimes}(\Gamma_{-1}(\mathcal{H}_B),\tau_B))\] as the following
\[\mathcal{S}_1(\cdots\otimes(x_0\otimes y_0)\otimes(x_1\otimes y_1)\otimes\cdots)\colon=\cdots\otimes(x_{-1}\otimes y_0)\otimes(x_0\otimes y_1)\otimes\cdots\] and similarly, \[\mathcal{S}_2(\cdots\otimes(x_0\otimes y_0)\otimes(x_1\otimes y_1)\otimes\cdots)\colon=\cdots\otimes(x_0\otimes y_{-1})\otimes(x_1\otimes y_0)\otimes\cdots.\] Now let us define the linear maps on $\mathcal{M}$ as
\[U_i(y)\colon=d_i((I_{B(\ell_I^2)}\otimes\mathcal{S}_i)(y))d_i\] for $i=1,2.$ Clearly, each $U_i$ is a unital normal trace preserving $*$-automorphism of $\mathcal M$ for $i=1,2.$ To avoid notational complexity, let us introduce the following \[W^A_{ij}=w_A([e_i]_A)w_A([e_j]_A)
\otimes I,\]\[ W^B_{ij}=I\otimes w_B([e_i]_B)w_B([e_j]_B).\]   and \[W_{ij}^{A,B}=w_A([e_i]_A)w_A([e_j]_A)
\otimes w_B([e_i]_B)w_B([e_j]_B).\]
We first check that $(U_1,U_2)$ is a commuting couple. For this, let us take an element \[y=a\otimes\cdots\underbrace{\otimes}_0(x_0\otimes y_0)\otimes(x_1\otimes y_1)\otimes\cdots\in\mathcal M,\] and compute
\begin{eqnarray}\label{rev1}
\nonumber U_2(y)&=& d_2((I_{B(\ell_I^2)}\otimes \mathcal{S}_2)(a\otimes\cdots\underbrace{\otimes}_0(x_0\otimes y_0)\otimes(x_1\otimes y_1)\otimes\cdots))d_2\\\nonumber
&=&d_2(a\otimes\cdots\underbrace{\otimes}_0(x_0\otimes y_{-1})\otimes(x_1\otimes y_0)\otimes\cdots)d_2\\\nonumber
&=&\sum_{j,k}e_{jj}ae_{kk}\otimes\cdots \underbrace{\otimes}_0(x_0\otimes w_B([e_j]_B)y_{-1}w_B([e_k]_B))\otimes(x_1\otimes y_0)\otimes\cdots.
\end{eqnarray}
Therefore, by above we have that $U_1U_2(y)$ is also equal to
\begin{eqnarray}
\nonumber\sum_{j,k}d_1((I_{B(\ell_I^2)}\otimes\mathcal{S}_1)(e_{jj}ae_{kk}
\otimes\cdots\underbrace{\otimes}_0(x_0\otimes w_B([e_j]_B)y_{-1}w_B([e_k]_B))\otimes(x_1\otimes y_0)\otimes\cdots))d_1\\\nonumber
=\sum_{j,k}d_1(e_{jj}ae_{kk}\otimes\cdots\underbrace{\otimes}_0(x_{-1}\otimes w_B([e_j]_B)y_{-1}w_B([e_k]_B))\otimes(x_0\otimes y_0)\otimes\cdots)d_1\\\nonumber
=\sum_{r,s,j,k}e_{rr}e_{jj}ae_{kk}e_{ss}\otimes\cdots\underbrace{\otimes}_0
(w_A([e_r]_A)x_{-1}w_A([e_s]_A)\otimes w_B([e_j]_B)y_{-1}w_B([e_k]_B))\otimes(x_0\otimes y_0)\otimes\cdots\\\nonumber
=\sum_{r,s}e_{rr}ae_{ss}\otimes\cdots\underbrace{\otimes}_0(w_A([e_r]_A)x_{-1}w_A([e_s]_A)\otimes w_B([e_r]_B)y_{-1}w_B([e_s]_B))\otimes(x_0\otimes y_0)\otimes\cdots.
\end{eqnarray}
Similarly, we see that 
\begin{eqnarray}\label{rev2}
\nonumber U_1(y)&=& d_1((I_{B(\ell_I^2)}\otimes \mathcal{S}_1)(a\otimes\cdots\underbrace{\otimes}_0(x_0\otimes y_0)\otimes(x_1\otimes y_1)\otimes\cdots))d_1\\\nonumber
&=&d_1(a\otimes\cdots\underbrace{\otimes}_0(x_{-1}\otimes y_{0})\otimes(x_0\otimes y_1)\otimes\cdots)d_1\\\nonumber
&=&\sum_{j,k}e_{jj}ae_{kk}\otimes\cdots \underbrace{\otimes}_0(w_A([e_j]_A)x_{-1}w_A([e_k]_A)\otimes y_0)\otimes(x_0\otimes y_1)\otimes\cdots.
\end{eqnarray}
Therefore, we have $U_2U_1(y)$ is equal to
\begin{eqnarray*}
\nonumber\sum_{j,k}d_2((I_{B(\ell_I^2)}\otimes\mathcal{S}_2)(e_{jj}ae_{kk}\otimes\cdots \underbrace{\otimes}_0(w_A([e_j]_A)x_{-1}w_A([e_k]_A)\otimes y_0)\otimes(x_0\otimes y_1)\otimes\cdots))d_2\\\nonumber
=\sum_{j,k}d_2(e_{jj}ae_{kk}\otimes\cdots\underbrace{\otimes}_0(w_A([e_j]_A)x_{-1}w_A([e_k]_A)\otimes y_{-1})\otimes(x_0\otimes y_0)\otimes\cdots)d_2\\\nonumber
=\sum_{r,s,j,k}e_{rr}e_{jj}ae_{kk}e_{ss}\otimes\cdots\underbrace{\otimes}_0
(w_A([e_j]_A)x_{-1}w_A([e_k]_A)\otimes w_B([e_r]_B)y_{-1}w_B([e_s]_B))\otimes(x_0\otimes y_0)\otimes\cdots\\\nonumber
=\sum_{r,s}e_{rr}ae_{ss}\otimes\cdots\underbrace{\otimes}_0(w_A([e_r]_A)x_{-1}w_A([e_s]_A)\otimes w_B([e_r]_B)y_{-1}w_B([e_s]_B))\otimes(x_0\otimes y_0)\otimes\cdots.
\end{eqnarray*}
Thus, we have that $(U_1,U_2)$ is a commuting couple.

We claim that
\begin{itemize}
 \item[(1)]$U_1^kU_2^l(J(x))=\sum_{i,j}x_{ij}e_{ij}\otimes
\cdots\underbrace{\otimes}_0(W_{ij}^{A,B})^{\otimes l}\otimes(W^A_{ij})^{\otimes(k-l)}\otimes(I\otimes I)\otimes\cdots,$ for all positive integers $k,l$
 with $k>l,$ where we denote $a^{\otimes l}\colon=\underbrace{a\otimes\cdots  \otimes a}_l.$
 \item[(2)]$U_1^kU_2^l(J(x))=\sum_{i,j}x_{ij}e_{ij}\otimes\cdots\
 \underbrace{\otimes}_0(W_{ij}^{A,B})^{\otimes k}\otimes(W^B_{ij})^{\otimes(l-k)}\otimes(I\otimes I)\otimes\cdots,$ for all integers $k,l$
 with $k<l.$
 \item[(3)]$U_1^kU_2^k(J(x))=\sum_{i,j}x_{ij}e_{ij}\otimes\cdots
 \underbrace{\otimes}_0(W_{ij}^{A,B})^{\otimes k}\otimes(I\otimes I)\otimes\cdots,$ for all integers $k.$
\end{itemize}
We shall only prove (1). We proceed by induction. Assuming that the identity is true for some $k>l,$ we have that the quantity $U_1^{k+1}U_2^l(J(x))$ is equal to the following
\begin{eqnarray*}
&& d_1\Big(\Big(I_{B(\ell_I^2)}\otimes\mathcal{S}_1\Big)
\Big(\sum_{i,j}x_{ij}e_{ij}\otimes\cdots\underbrace{\otimes}_0
(W_{ij}^{A,B})^{\otimes l}\otimes(W^A_{ij})^{\otimes(k-l)}\otimes(I\otimes I)\otimes\cdots\Big)\Big)d_1\\
 &=&\Big(\sum_{r\in I}e_{rr}\otimes\cdots\underbrace{\otimes}_0(w_A([e_r]_A)\otimes I)\otimes(I\otimes I)\otimes\cdots\Big)\Big(\Big(I_{B(\ell_I^2)}\otimes\mathcal{S}_1\Big)\\
 &&\Big(\sum_{i,j}x_{ij}e_{ij}\otimes\cdots\underbrace{\otimes}_0(W_{ij}^{A,B})^{\otimes l}\otimes(W^A_{ij})^{\otimes(k-l)}\otimes(I\otimes I)\otimes\cdots\Big)\Big)\\
&&\Big(\sum_{s\in I}e_{ss}\otimes\cdots\underbrace{\otimes}_0(w_A([e_s]_A)\otimes I)\otimes(I\otimes I)\otimes\cdots\Big)\\
 &=&\sum_{i,j,r,s}x_{ij}e_{ij}e_{rr}e_{ss}\otimes\cdots\underbrace
 {\otimes}_0\Big((w_A([e_r]_A)w_A([e_s]_A)
 \otimes w_B([e_i]_B)w_B([e_j]_B)\Big)\otimes\Big((W_{ij}^{A,B})^{\otimes l}\\
 &&\otimes(W^A_{ij})^{\otimes(k-l)}\otimes(I\otimes I)\otimes\cdots\Big)\\
 &=&\sum_{i,j}x_{ij}e_{ij}\otimes\cdots\underbrace{\otimes}_0(W_{ij}^{A,B})^{\otimes l}\otimes(W^A_{ij})^{\otimes(k+1-l)}\otimes(I\otimes I)\otimes\cdots.
\end{eqnarray*}
For $k>l,$ we observe the following
 \begin{eqnarray*}
 EU_1^kU_2^lJ(x)&=&\Big(I_{B(\ell_I^2)}\otimes\cdots\otimes
 (\tau_A\otimes\tau_B)\otimes\cdots\Big)\\ &&\Big(\sum_{i,j}x_{ij}e_{ij}\otimes\cdots\underbrace{\otimes}_0(W_{ij}^{A,B})^{\otimes l}\otimes(W^A_{ij})^{\otimes(k-l)}\otimes(I\otimes I)\otimes\cdots\Big)\\
 &=&\sum_{i,j}\tau_A(\omega_A([e_i]_A)\omega_A([e_j]_A))^k\tau_B
 (\omega_B([e_i]_B)\omega_B([e_j]_B))^lx_{ij}e_{ij}\\
 &=&\sum_{ij}\big(\langle e_i,e_j\rangle_A\big)^k\big(\langle e_i,e_j\rangle_B\big)^lx_{ij}e_{ij}\\
 &=&\sum_{ij}a_{ij}^kb_{ij}^lx_{ij}e_{ij}\\
 &=&M_A^kM_B^l(x).
 \end{eqnarray*}
 This completes the proof.
 \end{proof}
We present the proof of Theorem \ref{GVNA} only for two unital completely positive Fourier $(M_t,M_s)$ associated to real functions $t,s\colon G\to\mathbb R.$

{{\bf{Proof of Theorem \ref{GVNA}:}}}
\begin{proof}
Following \cite{AR13} (Proof of Theorem 4.6 ), since $M_t\colon VN(G)\to VN(G)$ is completely positive, one can define a positive symmetric bilinear form $\langle , \rangle_{\ell^{2,t}}$ on the real span of $e_g,$ $g\in G,$ $(e_g)_{g\in G}$ being the standard basis of $\ell^2_G,$ as \[\langle e_g, e_h\rangle_{\ell^{2,t}}\colon=t_{g^{-1}h}\] for $g,h\in G.$ Denote $\ell^{2,t}$ to be the completion real pre-Hilbert space after quotienting by the associated kernel. Similarly, define $\ell^{2,s}$ to be the real Hilbert space corresponding to the Fourier multiplier $s.$
	For all $g\in G,$ consider the unital trace preserving $*$-automorphism \[\alpha(g)\colon\Gamma_{-1}(\ell^{2,t}\otimes_2\ell_{\mathbb Z}^2)\overline{\otimes}\Gamma_{-1}(\ell^{2,s}\otimes_2\ell_{\mathbb Z}^2)\to \Gamma_{-1}(\ell^{2,t}\otimes_2\ell_{\mathbb Z}^2)\overline{\otimes}\Gamma_{-1}(\ell^{2,s}\otimes_2\ell_{\mathbb Z}^2)\] defined as \[\alpha(g)(w(h\otimes v)\otimes \widetilde{w}(\widetilde{h}\otimes\widetilde{v}))\colon=w(gh\otimes v)\otimes \widetilde{w}(g\widetilde{h}\otimes\widetilde{v}).\] From the dynamical system $(\Gamma_{-1}(\ell^{2,t}\otimes_2\ell_{\mathbb Z}^2)\overline{\otimes}\Gamma_{-1}(\ell^{2,s}\otimes_2\ell_{\mathbb Z}^2), G,\alpha)$ we define the crossed product \[\mathcal M=(\Gamma_{-1}(\ell^{2,t}\otimes_2\ell_{\mathbb Z}^2)\overline{\otimes}\Gamma_{-1}(\ell^{2,s}\otimes_2\ell_{\mathbb Z}^2))\rtimes _\alpha G.\]
We can identify $\Gamma_{-1}(\ell^{2,t}\otimes_2\ell_{\mathbb Z}^2)\overline{\otimes}\Gamma_{-1}(\ell^{2,s}\otimes_2\ell_{\mathbb Z}^2)$ 
as a subalgebra of $\mathcal M$. Let $J$ be the canonical normal unital injective
$*$-homomorphism $\text{VN}(G)\to\mathcal M$
. We denote by $\tau$ and $\widetilde{\tau}$ the faithful finite normal traces on $\Gamma_{-1}(\ell^{2,t}\otimes_2\ell_{\mathbb Z}^2)$ and $\Gamma_{-1}(\ell^{2,s}\otimes_2\ell_{\mathbb Z}^2)$ respectively. Let $\tau_{\mathcal M}$ be the canonical trace on $\mathcal M.$ For all $x\in \Gamma_{-1}(\ell^{2,t}\otimes_2\ell_{\mathbb Z}^2)\overline{\otimes}\Gamma_{-1}(\ell^{2,s}\otimes_2\ell_{\mathbb Z}^2)$ we have $\tau_{\mathcal M}(xJ(\lambda(g)))=\delta_{g,e_G}\tau\otimes\widetilde{\tau}(x).$
We denote by $E\colon\mathcal M\to \text{VN}(G)$ to be the canonical faithful normal trace preserving conditional expectation. Define the following shift operators \[\mathcal S\colon\Gamma_{-1}(\ell^{2,t}\otimes_2\ell_{\mathbb Z}^2)\overline{\otimes}\Gamma_{-1}(\ell^{2,s}\otimes_2\ell_{\mathbb Z}^2)\to \Gamma_{-1}(\ell^{2,t}\otimes_2\ell_{\mathbb Z}^2)\overline{\otimes}\Gamma_{-1}(\ell^{2,s}\otimes_2\ell_{\mathbb Z}^2)\] as $\mathcal S(w(h\otimes e_n)\otimes y)\colon=(w(h\otimes e_{n+1}))\otimes y$ and 
\[\widetilde{\mathcal S}\colon\Gamma_{-1}(\ell^{2,t}\otimes_2\ell_{\mathbb Z}^2)\overline{\otimes}\Gamma_{-1}(\ell^{2,s}\otimes_2\ell_{\mathbb Z}^2)\to \Gamma_{-1}(\ell^{2,t}\otimes_2\ell_{\mathbb Z}^2)\overline{\otimes}\Gamma_{-1}(\ell^{2,s}\otimes_2\ell_{\mathbb Z}^2)\] as $\widetilde{\mathcal S}(x\otimes \widetilde w(\widetilde{h}\otimes e_{n}))\colon=x\otimes \widetilde w(\widetilde{h}\otimes e_{n+1}).$ Thereafter, define the following operators 
$U\colon\mathcal M\to\mathcal M$ as \[U(x\lambda(g))=w(e_G\otimes e_0)\otimes 1\mathcal {S}(x)J(\lambda(g))w(e_G\otimes e_0)\otimes 1\] and $\widetilde U\colon\mathcal M\to\mathcal M$ as \[\widetilde U(x\lambda(g))\colon=1\otimes \widetilde w(e_G\otimes e_0)\widetilde{\mathcal S}(x)J(\lambda(g))1\otimes \widetilde w(e_G\otimes e_0).\] As in \cite{AR13}, it is easy to check that $(U,\widetilde{U})$ is a commuting tuple of unital trace preserving $*$-automorphisms of $\mathcal M.$
One can easily prove by induction that 
\begin{eqnarray*}
U^k \widetilde{U}^l J(\lambda(g))&=&w(e_G\otimes e_0)w(e_G\otimes e_1)\dots w(e_G\otimes e_{k-1})w(g\otimes e_{k-1})\dots  w(g\otimes e_1)w(g\otimes e_0)\\
&&\otimes \widetilde{w}(e_G\otimes e_0)\widetilde{w}(e_G\otimes e_1)\dots \widetilde{w}(e_G\otimes e_{l-1})\widetilde{w}(g\otimes e_{l-1})\dots  \widetilde{w}(g\otimes e_1)\widetilde{w}(g\otimes e_0)J(\lambda(g)).
\end{eqnarray*}
Let $E\colon\mathcal M\to \text{VN}(G)$ be the canonical conditional expectation operator. Then we have by above and \cite{AR13}
\begin{eqnarray*}
	EU^k \widetilde{U}^l J(\lambda(g))&=&\tau\otimes \widetilde \tau \Big(w(e_G\otimes e_0)w(e_G\otimes e_1)\dots w(e_G\otimes e_{k-1})w(g\otimes e_{k-1})\dots  w(g\otimes e_1)w(g\otimes e_0)\\
	&&\otimes\widetilde{w}(e_G\otimes e_0)\widetilde{w}(e_G\otimes e_1)\dots \widetilde{w}(e_G\otimes e_{l-1})\widetilde{w}(g\otimes e_{l-1})\dots  \widetilde{w}(g\otimes e_1)\widetilde{w}(g\otimes e_0)\Big)\lambda(g)\\
	&=&\tau\Big(w(e_G\otimes e_0)w(e_G\otimes e_1)\dots w(e_G\otimes e_{k-1})w(g\otimes e_{k-1})\dots  w(g\otimes e_1)w(g\otimes e_0)\Big)\\
	&&\widetilde{\tau}\Big(\widetilde{w}(e_G\otimes e_0)\widetilde{w}(e_G\otimes e_1)\dots \widetilde{w}(e_G\otimes e_{l-1})\widetilde{w}(g\otimes e_{l-1})\dots  \widetilde{w}(g\otimes e_1)\widetilde{w}(g\otimes e_0)\Big)\lambda(g)\\
	&=&(t_g)^k(s_g)^l\lambda(g)\\
	&=&M_t^kM_s^l\lambda(g).
\end{eqnarray*}
Hence the theorem.
\end{proof}
We now prove a dilation theorem for multi-parameter $w^*$-semigroup of self adjoint unital completely positive Schur multipliers on $B(L^2(\Omega)).$ This theorem is a multivariate generalization of similar theorems proved in \cite{AR13}, \cite{AR19} and \cite{AR191}. Before presenting the statement of the theorem, we briefly discuss some necessary background. We refer \cite{AR13}, \cite{AR191} and references therein for more elaboration on the terminologies and the tools, we use here.

Suppose $H$ is a real Hilbert space. An $H$-isonormal process on a probability space $(\Omega_0,\mu)$ is a linear map $W\colon H\to L^0(\Omega_0)$ which satisfies the following properties:
\begin{enumerate}
	\item[(i)] For all $h\in H$ the random variable $W(h) $ is centred real Gaussian.
	\item[(ii)] For all $h_1,h_2\in H$ we have $\mathbb{E}(W(h_1)W(h_2))=\langle h_1,h_2\rangle_{H}.$
	\item[(iii)] The linear span of the products $\prod_{i=1}^mW(h_i)$ with $m\geq 0$ and $h_i\in H$ for $1\leq i\leq m$ is dense in real Hilbert space $L^2_\mathbb{R}(\Omega_0).$
\end{enumerate}
In above $L^0(\Omega_0)$ denotes the space of all real measurable functions on $\Omega_0$ and the case $m=0$ corresponds to the empty product and nothing but the constant function $1.$ The span of $e^{iW(h)}$ is actually a dense subset of $L^p(\Omega_0)$ for $1\leq p<\infty.$ We also have the identity \[\int_{\Omega}e^{-itW(h)}d\mu=e^{-\frac{t}{2}\|h\|_H^2}\] for $t\in\mathbb R$ and $h\in H.$

Let $\Omega$ be a $\sigma$-finite measure space. Define $S^\infty_\Omega\colon=S^\infty(L^2(\Omega)).$ For $f\in L^2(\Omega\times \Omega)$ define
\[K_f\colon L^2(\Omega)\to L^2(\Omega)\] as $K_f(\zeta)=\int_\Omega\zeta(s)f(s,.)ds.$ A measurable function $\phi\colon\Omega\times \Omega\to\mathbb C$ is called a Schur multiplier on $B(L^2(\Omega))$ if for any $f\in L^2(\Omega\times\Omega)$ we have $K_{\phi f}\in S^\infty_\Omega.$ It follows from the closed graph theorem that the map $K_f\mapsto K_{\phi f}$ is a bounded linear map from $S^\infty_\Omega$ to $S^\infty_\Omega.$ The second adjoint is a $w^*$-continuous map from $B(L^2(\Omega))$ to $B(L^2(\Omega))$ which is denoted by $M_\phi$ and called the Schur multiplier on $B(L^2(\Omega)).$ For any $g\in L^\infty(\Omega)$ we denote $\mathscr{L}_g\colon L^2(\Omega)\to L^2(\Omega)$ to be the usual multiplication operator defined by $\mathscr{L}_g\zeta=g\zeta.$ Equipped with this, we state our theorem. For simplicity, we deliver a proof for $n=2.$
\begin{thm}Let $(T_{t_1,t_2,\dots,t_n})_{t_1,t_2,\dots,t_\geq 0}$ be an $n$-parameter $w^*$-semigroup of self adjoint unital completely positive Schur multipliers on $B(L^2(\Omega)).$ Then, there exists a hyperfinite von Neumann algebra $\mathcal{M}$ equipped with a normal semifinite faithful trace, a $w^*$-semigroup $(U_{t_1,t_2,\dots,t_n})_{t_1,t_2,\dots,t_\geq 0}$ of unital trace preserving $*$-automorphisms on $\mathcal{M},$ a
	unital trace preserving one-to-one normal $*$-homomorphism $J\colon B(L^2(\Omega))\to\mathcal{M}$ such that \[T_{t_1,t_2.\dots,t_n}=EU_{t_1,t_2.\dots,t_n}J,\] for all $t_i\geq 0,$ $1\leq i\leq n,$ where $E\colon \mathcal{M}\to B(L^2(\Omega))$ is the canonical faithful normal trace preserving conditional expectation operator associated with $J.$
\end{thm}
\begin{proof}
 Note that $(T_{t_1,0})_{t\geq 0}$ is one parameter $w^*$-continuous semigroup of self adjoint unital completely positive Schur multipliers on $B(L^2(\Omega)).$ Therefore, by \cite[Theorem 3.3]{AR191}, there exists a real Hilbert space $H_1$ and a measurable map $\alpha\colon\Omega\to H_1$ such that the Schur multiplier $T_{t_1,0}$ is associated with the symbol \[g^1_{t_1}(s,r)=e^{-t_1\|\alpha_s-\alpha_r\|_{H_1}^2}.\] Similarly, for the semigroup $(T_{0,t_2})_{t_2\geq 0}$ we find a real Hilbert space $H_2$ and a measurable map $\beta\colon\Omega\to H_2$ such that the Schur multiplier $T_{0,t_2}$ is associated with the symbol \[g^2_{t_2}(s,r)=e^{-t_2\|\beta_s-\beta_r\|_{H_2}^2}.\]
 Denote $\mathcal M\colon=L^\infty(\Omega_0\times\Omega_0)\overline{\otimes}B(L^2(\Omega)).$ Clearly, $\mathcal M$ is a hyperfinite von Neumann algebra equipped with the normal semifinite faithful trace $\tau_{\mathcal M}\colon=\int_{\Omega_0\times \Omega_0}\otimes Tr.$ Then, as in \cite{AR191}, $\mathcal M$ is $*$-isomorphic to $L^\infty(\Omega_0\times \Omega_0,B(L^2(\Omega)).$ Let us define the canonical normal unital trace preserving $*$-homomorphism
 \[J\colon B(L^2(\Omega))\to L^\infty(\Omega_0\times \Omega_0,B(L^2(\Omega))\] as the following \[J(x)\colon=1\otimes x.\] We denote by $E\colon L^\infty(\Omega_0\times \Omega_0,B(L^2(\Omega))\to B(L^2(\Omega))$ to be the canonical faithful normal trace preserving conditional expectation of $\mathcal M$ to $B(L^2(\Omega)).$ For any $(\omega_1,\omega_2)\in \Omega_0\times \Omega_0,$ and $t_1,t_2\geq 0,$ we define \[k^1_{t_1,\omega_1}(s)\colon=e^{\sqrt{2}it_1(W(\alpha_s))(\omega_1)}\] and 
 \[k^2_{t_2,\omega_2}(s)\colon=e^{\sqrt{2}it_2(W(\beta_s))(\omega_2)}.\] For any $t_1,t_2\geq 0,$ define $D_{t_1,t_2}\in L^\infty(\Omega_0\times \Omega_0, B(L^2(\Omega))$ by \[D_{t_1,t_2}(\omega_1,\omega_2)\colon=\mathscr{L}_{k^1_{t_1,\omega_1}k^2_{t_2,\omega_2}}.\] Clearly, $D_{t_1,t_2}$ is a unitary element of $L^\infty(\Omega_0\times\Omega_0,B(L^2(\Omega)).$ For any $t_1,t_2\geq 0$ define \[U_{t_1,t_2}\colon L^\infty(\Omega_0\times\Omega_0,B(L^2(\Omega))\to L^\infty(\Omega_0\times\Omega_0,B(L^2(\Omega))\] as the following \[U_{t_1,t_2}g\colon=D_{t_1,t_2}gD_{t_1,t_2}^*.\] Then $(U_{t_1,t_2})_{t_1,t_2\geq 0}$ is a $w^*$-continuous semigroup of trace preserving $*$-automorphisms on $\mathcal M.$ Now for any $K_f$ and almost every $r$ we have
 \begin{eqnarray*}
 \begin{split}
 (EU_{t_1,t_2}J(K_f))(\zeta)(r)&=(EU_{t_1,t_2}(1\otimes K_f))(\zeta)(r) \\
 &=(E(D_{t_1,t_2}(1\otimes K_f)D_{t_1,t_2}^*))(\zeta)(r) \\
 &=\int_{\Omega_0}\int_{\Omega_0}(D_{t_1,t_2}(\omega_1,\omega_2)K_fD_{t_1,t_2}
 (\omega_1,\omega_2)^*(\zeta))(r)d\mu(\omega_1)d\mu(\omega_2)\\
 &=\int_{\Omega_0\times\Omega_0}\int_\Omega k^1_{t_1,\omega_1}(s)k^2_{t_2,\omega_2}(s)\zeta(s)f(s,r)\overline{k^1_{t_1,\omega_1}(r)k^2_{t_2,\omega_2}(r)}ds d\mu(\omega_1)d\mu(\omega_2)\\
 &=\int_{\Omega}(\int_{\Omega_0\times\Omega_0}e^{\sqrt{2}it_1(W(\alpha_s-\alpha_r)(\omega_1))})e^{\sqrt{2}it_2(W(\beta_s-\beta_r)(\omega_2))}d\mu(\omega_1)d\mu(\omega_2))f(s,r)ds\\
 &=\int_\Omega\zeta(s)e^{-t_1\|\alpha_s-\alpha_r\|_{H_1}^2}e^{-t_2\|\beta_s-\beta_r\|_{H_2}^2}f(s,r)ds\\
 &=\int_\Omega\zeta(s)g^1_{t_1}(s,r)g^2_{t_2}(s,r)f(s,r)ds\\
 &=(K_{g^1_{t_1}g^2_{t_2}f}(\zeta))(r)\\
 &=((T_{t_1,t_2}(K_f))(\zeta))(r).
 \end{split}
 \end{eqnarray*}
The proof is completed by $w^*$-density as in \cite{AR191}.
\end{proof}
\begin{rem}\label{FREM}
Note that, one can easily prove a multivariate analogue of Corollary 4.3, Corollary 4.5 and Corollary 4.7 of \cite{AR13}. Also, using Ando's dilation theorem \cite{AN63}, Proposition 4.9 of \cite{AR13} is true for two variables.	
\end{rem}
\textbf{Acknowledgement:} The author thanks his thesis supervisor Prof. Parasar Mohanty for many stimulating discussions. He also thanks Guixang Hong, Christian Le Merdy, Gadadhar Misra and C. Arhancet for some suggestions. We also thank the referee for several constructive suggestion which significantly improved the presentation of the paper.

\end{document}